\documentclass{amsart}
\usepackage{amsmath, amssymb, amscd}

\newtheorem{theorem}{Theorem}[section]
\newtheorem{proposition}[theorem]{Proposition}
\newtheorem{lemma}[theorem]{Lemma}
\newtheorem{corollary}[theorem]{Corollary}

\theoremstyle{definition}
\newtheorem{definition}[theorem]{Definition}
\newtheorem{example}[theorem]{Example}
\newtheorem{remark}[theorem]{Remark}

\numberwithin{equation}{section}

\newcommand{\N}{\mathbb{N}}                        
\newcommand{\C}{\mathbb{C}}                        
\newcommand{\vp}{\varphi}                          
\newcommand{\VV}{\mathcal{V}}                      
\newcommand{\FF}{\mathcal{F}}                      
\newcommand{\OO}{\mathcal{O}}                      
\newcommand{\OXxi}{{\mathcal{O}}_{X,\xi}}          
\newcommand{\OYeta}{{\mathcal{O}}_{Y,\eta}}        
\newcommand{\Xxi}{X_{\xi}}                         
\newcommand{\Yeta}{Y_{\eta}}                       
\newcommand{\Zzeta}{Z_{\zeta}}                     
\newcommand{\Xpn}{X^{\{n\}}}                       
\newcommand{\xipn}{{\xi}^{\{n\}}}                  
\newcommand{\vpxi}{\vp_{\xi}}                      
\newcommand{\vpn}{\vp^{\{n\}}}                     
\newcommand{\antens}{\tilde{\otimes}}              
\newcommand{\tensR}{\tilde{\otimes}_R}             
\newcommand{\FpnR}{F^{\tilde{\otimes}^n_R}}        
\newcommand{\pp}{\mathfrak{p}}                     
\newcommand{\Ann}{\mathrm{Ann}}                    
\newcommand{\Spec}{\mathrm{Spec}\,}                
\newcommand{\Specan}{\mathrm{Specan}\,}            

\begin{document}

\title{Flatness testing over singular bases}

\author{Janusz Adamus}
\address{Department of Mathematics, The University of Western Ontario, London, Ontario, Canada N6A 5B7
         -- and -- Institute of Mathematics, Faculty of Mathematics and Computer Science,
         Jagiellonian University, ul. {\L}ojasiewicza 6, 30-348 Krak{\'o}w, Poland}
\email{jadamus@uwo.ca}
\author{Hadi Seyedinejad}
\address{Department of Mathematics, The University of Western Ontario, London, Ontario, Canada N6A 5B7}
\email{sseyedin@uwo.ca}
\thanks{J. Adamus's research was partially supported by Natural Sciences and Engineering 
Research Council of Canada Discovery Grant 355418-2008 and Polish Ministry of 
Science Discovery Grant N N201 540538}


\begin{abstract}
We show that non-flatness of a morphism $\vp:X\to Y$ of complex-analytic spaces with a locally irreducible target of dimension $n$ manifests in the existence of vertical components in the $n$-fold fibred power of the pull-back of $\vp$ to the desingularization of $Y$.
An algebraic analogue follows: Let $R$ be a locally (analytically) irreducible finite type $\C$-algebra and an integral domain of Krull dimension $n$, and let $S$ be a regular $n$-dimensional algebra of finite type over $R$ (but not necessarily a finite $R$-module), such that $\Spec S\to\Spec R$ is dominant. Then a finite type $R$-algebra $A$ is $R$-flat if and only if $(A^{\otimes^n_R})\otimes_RS$ is a torsion-free $R$-module.
\end{abstract}
\maketitle

{\hfill\textit{Dedicated to Edward Bierstone and Pierre D. Milman\ }}

{\hfill\textit{for their 65'th birthdays.}}
\bigskip

\section{Introduction and main results}
\label{sec:intro}

This note is concerned with the algebraic notion of flatness. The past few years have seen a significant progress in the attempts to understand flatness of an (analytic or algebraic) morphism in explicit geometric terms. Following the work of Galligo and Kwieci{\'n}ski \cite{GK}, Adamus, Bierstone and Milman \cite{ABM1} characterized flatness of a morphism $\vp:X\to Y$ of complex-analytic spaces in terms of the so-called vertical components (see~\ref{subsec:vertical}) in fibred powers of $\vp$. An analogous result in the algebraic category was later obtained by Avramov and Iyengar \cite{AI}. Both theorems assume smoothness of the target $Y$, as the arguments, inspired by Auslander's~\cite{Au}, rely on homological properties of modules over regular local rings. 

On the other hand, very little is known in this context in case of singular targets. The best result we know of works only for finite maps over irreducible plane curves (see~\cite{HW}).

Our goal in the present paper is to generalize the results of \cite{ABM1} to the singular setting. The following is a special case of our main result (Theorem~\ref{thm:main} and Corollary~\ref{cor:main} below).

\begin{theorem}
\label{thm:1}
Let $\vp:X\to Y$ be a morphism of complex-analytic spaces, where $Y$ is locally irreducible. Let $\xi\in X$. Let $\Xpn$ denote the $n$-fold fibred power of $X$ over $Y$, where $n=\dim Y$, and let $\xipn\in\Xpn$ be the diagonal point corresponding to $\xi$. Finally, let $\sigma_\zeta:\Zzeta\to\Yeta$ be any dominant morphism of complex-analytic space germs with $\Zzeta$ smooth of dimension $n$ (e.g., $\Zzeta$ can be a desingularization of $\Yeta$), where $\eta=\vp(\xi)$. Then $\vp$ is not flat at $\xi$ if and only if the fibred product $\Xpn\times_YZ$ has a vertical component at $(\xipn,\zeta)$; i.e., a local irreducible component (perhaps embedded) of $\Xpn\times_YZ$ at $(\xipn,\zeta)$ whose image by the canonical map $\Xpn\times_YZ\to Y$ is nowhere-dense in $Y$.
\end{theorem}

Throughout the paper, by a \emph{dominant morphism $\sigma_\zeta:\Zzeta\to\Yeta$ of complex-analytic space germs} we mean a germ at $\zeta$ of a holomorphic mapping $\sigma:Z\to Y$ such that $\sigma(\zeta)=\eta$ and, for an arbitrarily small open neighbourhood $U$ of $\zeta$ in $Z$, we have $\dim\sigma(U)=\dim\Yeta$.

\begin{remark}
\label{rem:generalization}
Note that, according to Theorem~\ref{thm:1}, $(\Xpn\times_YZ)_{(\xipn\!,\zeta)}$ has no vertical components for \emph{some} map germ $\sigma_\zeta:\Zzeta\to\Yeta$ satisfying the hypotheses of the theorem if and only if it has no vertical components for \emph{all} such map germs.
In particular, if $Y$ is smooth then, to investigate flatness of $\vp_\xi$, we can choose $Z=Y$ and $\sigma=\mathrm{id}_Y$. Then $\Xpn\times_YZ$ can be identified with $\Xpn$, and so in this case Theorem~\ref{thm:1} specializes to \cite[Thm.\,1.1]{ABM1}.
\end{remark}
\medskip

Theorem~\ref{thm:main} and Corollary~\ref{cor:main} provide a more general criterion for $\OYeta$-flatness of a finitely generated $\OXxi$-module. Before formulating these results, we will recall the notion of verticality in complex-analytic modules.

\subsection{Background: vertical elements and analytic tensor product}
\label{subsec:vertical}

Let $\vp:X\to Y$ be a holomorphic mapping of complex-analytic spaces. Let $\xi\in X$, $\eta=\vp(\xi)\in Y$, and let $\vpxi:\Xxi\to\Yeta$ be the induced morphism of germs. Let $\Sigma_{\xi}$ be an irreducible component (isolated or embedded) of $\Xxi$; i.e., the zero-set germ of an associated prime (minimal or not) of the zero ideal in the local ring $\OXxi$. We will say that $\Sigma_{\xi}$ is a \emph{geometric vertical component} with respect to $\vpxi$, when a sufficiently small representative of $\Sigma_{\xi}$ is mapped by a representative of $\vpxi$ into a nowhere-dense (with respect to the strong topology) subset of a representative of $\Yeta$.

\begin{remark}
\label{rem:well-defined}
One also defines another kind of vertical components, called \emph{algebraic vertical} (see, e.g., \cite{ABM1}). We will however not consider them at all in the present paper, and hence will simply use the term \emph{vertical components} when refering to the geometric vertical ones defined above.
\end{remark}

Since the set of zero-divisors of the ring $\OXxi$ coincides with the union of its associated primes, it follows that $\vpxi:\Xxi\to\Yeta$ has a vertical component if and only if there exists a nonzero element $m\in\OXxi$ such that (a sufficiently small representative of) the zero-set germ $\VV(\Ann_{\OXxi}(m))$ of the annihilator of $m$ in $\OXxi$ is mapped by (a representative of) $\vpxi$ to a nowhere-dense subset of (a representative of) $\Yeta$. We can thus extend the notion of vertical component to a finitely generated $\OXxi$-module $F$:

\begin{definition}
\label{def:vert-element}
Let $I:=\Ann_{\OXxi}(F)$ and let $Z_{\xi}$ be the germ of a complex analytic subspace of $X$, defined by $\OO_{Z,\xi}:=\OXxi/I$.
We say that $F$ has a \emph{vertical component} over $Y_{\eta}$ (or over $\OYeta$), when $Z_{\xi}$ has a vertical component over $Y_{\eta}$ in the above sense; equivalently, there exists a nonzero $m\in F$ such that the $\VV(\Ann_{\OXxi}(m))$ is mapped to a nowhere-dense subgerm of $Y_{\eta}$. We will call such $m$ a \emph{geometric vertical element} (or simply a \emph{vertical element}) of $F$ over $Y_{\eta}$ (or over $\OYeta$).
\end{definition}

Vertical elements are geometric analogues of zero-divisors in commutative algebra (see~\cite[\S\,1]{ABM1} for details).

\begin{remark}
\label{rem:vert-equivalence}
In the special case that $F = \OXxi$, $X_\xi$ has no vertical components over $Y_\eta$ if and only if $\OXxi$ (as an $\OXxi$-module) has no vertical elements over $\OYeta$.
\end{remark}

Now let $R$ denote a local analytic $\C$-algebra, that is, a quotient ring $\C\{y\}/J = \C\{y_1,\dots,y_n\}/J$, for some $n\in\N$ and an ideal $J$ in $\C\{y\}$ (the ring of convergent power series in $n$ complex variables). 
By a \emph{local analytic $R$-algebra} we mean a ring of the form $R\{x\}/I:=\C\{y,x\}/I$, with the canonical homomorphism $R \to R\{x\}/I$,
where $I$ is an ideal in $\C\{y,x\} = \C\{y_1,\dots,y_n,x_1,\dots,x_m\}$ containing $J\cdot\C\{x,y\}$. Let $F$ denote an $R$-module.
Following \cite{GR}, we call $F$ an \emph{analytic $R$-module}\footnote{Such $F$ is also called an \emph{almost finitely generated $R$-module}, after \cite{GK}.} when $F$ is a finitely generated $A$-module, for some local analytic $R$-algebra $A$.
In this case, there is a morphism of germs of complex-analytic spaces $\vpxi:\Xxi\to\Yeta$ such that $R \cong \OYeta$, $A \cong \OXxi$, 
$R \to A$ is the induced pull-back homomorphism $\vp^*_{\xi}: \OYeta \to \OXxi$, and $F$ is a finitely generated $\OXxi$-module. We say that a nonzero element $m\in F$ is \emph{vertical} over $R$ if $m$ is vertical over $\OYeta$ in the sense of Definition~\ref{def:vert-element}.
It is easy to see that the notion of vertical element is well-defined; i.e., independent of the choice of a witness $A$ for $F$ (cf. \cite{GK}).
\medskip

In the category of local analytic $R$-algebras with local $R$-algebra homomorphisms, the coproduct of two objects $A=R\{x\}/I$, $B=R\{x'\}/I'$ exists, called the \emph{analytic tensor product} over $R$, and is denoted by $A\tensR B$. This can be shown to be isomorphic to $\frac{R\{x,x'\}}{I+I'}$ (cf. \cite{GR}).
Given two analytic $R$-modules $F_1$ and $F_2$, witnessed by $A_1$ and $A_2$ respectively, one defines their analytic tensor product over $R$ as
\[
F_1\tensR F_2:=\left((F_1\otimes_{A_1}(A_1\tensR A_2)\right)\otimes_{A_1\tensR A_2}\left((A_1\tensR A_2)\otimes_{A_2}F_2\right)\,.
\]

Recall that every local analytic $\C$-algebra $R$ corresponds to a unique complex-analytic space germ denoted $\Specan R$.
The duality between the categories of complex-analytic germs and local analytic $\C$-algebras (see \cite{F}) implies the following important isomorphism.
Suppose that $\vp_1:X_1\to Y$ and $\vp_2:X_2\to Y$ are holomorphic mappings of analytic spaces, with $\vp_1(\xi_1)=\vp_2(\xi_2)=\eta$. Then the local rings $\OO_{X_i,\xi_i}$ ($i=1,2$) are $\OYeta$-modules and, by the uniqueness of fibred product and of analytic tensor product, the local ring $\OO_{Z,(\xi_1,\xi_2)}$ of the fibred product $Z=X_1\times_YX_2$ at $(\xi_1,\xi_2)$ is canonically isomorphic to $\OO_{X_1,\xi_1}\antens_{\OYeta}\OO_{X_2,\xi_2}$.
Therefore, given a holomorphic germ $\vp_{\xi}:X_{\xi}\to Y_{\eta}$, we can identify the $d$-fold analytic tensor power $\OXxi^{{\antens}^d_{\OYeta}}=\OXxi\antens_{\OYeta}\dots\antens_{\OYeta}\OXxi$ with the local ring of the $d$-fold fibred power $\OO_{X^{\{d\}},\xi^{\{d\}}}$, for $d\geq1$.

\subsection{Main results}
Our main theorem is the following flatness criterion.

\begin{theorem}
\label{thm:main}
Let $R$ be a local analytic $\C$-algebra and an integral domain of dimension $n$. Let $F$ be an analytic $R$-module, and let $S$ be any local analytic $R$-algebra which is regular, $n$-dimensional, and such that the induced morphism of complex-analytic space germs $\Specan S\to\Specan R$ is dominant. Then $F$ is $R$-flat if and only if the analytic tensor product
\[
\underbrace{F\tensR\dots\tensR F}_{n\ \mathrm{times}}\tensR S
\]
has no vertical elements over $R$ (equivalently, over $S$).
\end{theorem}

\begin{remark}
\label{rem:desing-model}
Note that an $R$-algebra $S$ with the above properties always exists. One can, for instance, take $S$ to be the local ring of a desingularization of (a sufficiently small representative of) $\Specan R$ (cf.~\cite{BM2}).
\end{remark}

\begin{corollary}
\label{cor:main}
Let $\vp:X\to Y$ be a morphism of complex-analytic spaces, where $Y$ is locally irreducible, and let $\FF$ be a coherent $\OO_X$-module. Let $\xi\in X$. Let $\Xpn$ denote the $n$-fold fibred power of $X$ over $Y$, where $n=\dim Y$, and let $\xipn\in\Xpn$ be the diagonal point corresponding to $\xi$. Finally, let $\sigma_\zeta:\Zzeta\to\Yeta$ be any dominant morphism of complex-analytic space germs with $\Zzeta$ smooth of dimension $n$, where $\eta=\vp(\xi)$. Then $\FF_\xi$ is a flat $\OYeta$-module if and only if the analytic tensor product
\[
\underbrace{\FF_\xi\antens_{\OYeta}\dots\antens_{\OYeta}\FF_\xi}_{n\ \mathrm{times}}\antens_{\OYeta}\OO_{Z,\zeta}
\]
has no vertical elements over $\OYeta$ (equivalently, over $\OO_{Z,\zeta}$).
\end{corollary}

Now, Theorem~\ref{thm:1} follows from Corollary~\ref{cor:main} with $\FF=\OO_X$, according to Remark~\ref{rem:vert-equivalence} and the canonical isomorphism $\OXxi\antens_{\OYeta}\dots\antens_{\OYeta}\OXxi \cong \OO_{\Xpn,\xipn}$.

In case when the source $X$ is smooth, of the same dimension as $Y$, and the map germ $\vpxi$ is dominant, Theorem~\ref{thm:1} has a particularly pleasing form:

\begin{corollary}
\label{cor:X-smooth}
Let $\vpxi:\Xxi\to\Yeta$ be a dominant morphism of complex-analytic spaces germs, where $\Yeta$ is irreducible, $\Xxi$ is smooth, and $\dim\Xxi=\dim\Yeta=n$.
Then $\vpxi$ is flat if and only if the induced morphism $\vp^{\{n+1\}}_{\xi^{\{n+1\}}}:X^{\{n+1\}}_{\xi^{\{n+1\}}}\to\Yeta$ has no vertical components over $\Yeta$.
\end{corollary}

\begin{proof}
In Theorem~\ref{thm:1}, take $\Zzeta:=\Xxi$ and $\sigma_\zeta:=\vpxi$.
\end{proof}
\medskip

We derive Theorem~\ref{thm:main} from \cite[Thm.\,1.9]{ABM1} in Section~\ref{sec:proof-main}, via an analytic flatness descent (see below).
Section~\ref{sec:alg} contains an algebraic version of our flatness criterion (Theorem~\ref{thm:main-alg}) as well as Example~\ref{ex:sharp} proving sharpness of Theorem~\ref{thm:main}.

\section{Analytic flatness descent}
\label{sec:descent}

By definition, a ring homomorphism $R\to S$ descends flatness if, for any $R$-module $F$, flatness of $F\otimes_RS$ (as an $S$-module) implies flatness of $F$ (as an $R$-module). In general, flatness descent is a rare luxury. However, as we show below, it does hold for analytic modules over integral domains and local analytic $\C$-algebra homomorphisms inducing dominant morphisms of analytic space germs.

We will say that a homomorphism $R\to S$ of local analytic $\C$-algebras \emph{analytically descends flatness} if $S$-flatness of $F\tensR S$ implies $R$-flatness of $F$ for every analytic $R$-module $F$.

\begin{proposition}
\label{prop:descent}
Let $\kappa:R\to S$ be a homomorphism of local analytic $\C$-algebras, where $R$ is an integral domain. If the induced morphism $\Specan S\to\Specan R$ of complex-analytic space germs is dominant, then $\kappa$ analytically descends flatness.
\end{proposition}

\begin{proof}
For a proof by contradiction, suppose the morphism $\Specan S\to\Specan R$ is dominant and there exists a non-flat analytic $R$-module $F$ such that $F\tensR S$ is $S$-flat.
By Hironaka's local flattener theorem (see, e.g., \cite[Thm.\,7.12]{BM}), there exists a unique non-zero(!) ideal $P$ in $R$ such that $F\tensR R/P$ is $R/P$-flat and, for every morphism of analytic space germs $\vp:T\to\Specan R$, if $\OO_T\tensR F$ is $\OO_T$-flat then $\vp$ factors as
\[
T\to\Specan R/P\hookrightarrow\Specan R\,.
\]
Since $R$ is an integral domain, it follows that $\Specan R/P$ (hence also the image of $T$ via $\vp$) is nowhere-dense in $\Specan R$. Taking $T:=\Specan S$, we get a contradiction.
\end{proof}

\section{Proof of the main theorem}
\label{sec:proof-main}

We begin with a simple observation regarding commutativity of the analytic tensor product.
Let $R$ be a local analytic $\C$-algebra, let $A$ and $S$ be local analytic $R$-algebras, and let $F$ be a finitely generated $A$-module.
Set $A':=A\tensR S$ and $F':=F\tensR S$.

\begin{lemma}
\label{lem:commutativity}
With the above notation, for every $i\geq1$, we have:
\begin{itemize}
\item[(a)] $A^{\tensR^i}\tensR S\cong A'^{\antens_S^i}$ (as $R$-algebras).
\item[(b)] $F^{\tensR^i}\tensR S\cong F'^{\antens_S^i}$ (as $R$-modules).
\end{itemize}
\end{lemma}

\begin{proof}
We will prove the isomorphism (a) for $i=2$. The general case follows easily by induction. Write $S=R\{u\}/L$, $A=R\{x\}/I$, and $R\{x'\}/I'$ for another copy of $A$. Then, by~\ref{subsec:vertical}, we have
\begin{align*}
(A^{\tensR^2})\tensR S & = \left(\frac{R\{x\}}{I}\tensR\frac{R\{x'\}}{I'}\right)\tensR S\\
 & \cong \frac{R\{x,x'\}}{I+I'}\tensR\frac{R\{u\}}{L}\\
 & \cong \frac{R\{x,x',u\}}{I+I'+L}\\
 & \cong\frac{R\{x,u\}}{I+L}\antens_\frac{R\{u\}}{L}\frac{R\{x',u\}}{I'+L}\\
 & =A'^{\antens_S^2}\,.
\end{align*}
For the proof of part (b), apply (a) to a presentation of $F$ as a finite $A$-module.
\end{proof}
\medskip

For reader's convenience, we recall the analytic flatness criterion of \cite{ABM1}, on which the proof of Theorem~\ref{thm:main} is based.

\begin{theorem}[{\cite[Thm.\,1.9]{ABM1}}]
\label{thm:main-ABM1}
Let $S$ be a regular local analytic $\C$-algebra and let $M$ denote an analytic $S$-module. Let $n = \dim S$. Then
$M$ is $S$-flat if and only if the $n$-fold analytic tensor power $M^{\antens_S^n}$ has no vertical elements over $S$.
\end{theorem}

\subsection{Proof of Theorem~\ref{thm:main}}
\label{subsec:proof-main}

Let $S$ be an arbitrary local analytic $R$-algebra, which is a regular ring of dimension $\dim S=n=\dim R$, and such that the induced morphism $\Specan S\to\Specan R$ is dominant. Let $F$ be an arbitrary analytic $R$-module.

By Proposition~\ref{prop:descent} and since flatness is preserved by base change (\cite[Prop.\,6.8]{H}), it follows that $F$ is a flat $R$-module if and only if $F\tensR S$ is a flat $S$-module. Hence, by Theorem~\ref{thm:main-ABM1}, $F$ is $R$-flat if and only if the $n$-fold tensor power $(F\tensR S)^{\antens_S^n}$ has no vertical elements over $S$. By Lemma~\ref{lem:commutativity}, this is equivalent to saying that $\FpnR\tensR S$ has no vertical elements over $S$.

To complete the proof, it remains to show that the latter is equivalent to the lack of vertical elements in $\FpnR\tensR S$ over $R$. Let then $A$ denote a local analytic $R$-algebra for which $F$ is a finitely generated $A$-module, and let $\vpxi:\Xxi\to\Yeta$ denote a morphism of complex-analytic space germs such that $R=\OYeta$, $A=\OXxi$,
and $\vpxi^*:\OYeta\to\OXxi$ gives the $R$-algebra structure of $A$. Further, let $\sigma_\zeta:\Zzeta\to\Yeta$ denote a morphism of space germs such that $S=\OO_{Z,\zeta}$ and $\sigma_\zeta^*:\OYeta\to\OO_{Z,\zeta}$ gives the $R$-algebra structure of $S$. Consider the following fibred product square
\[
\begin{CD}
\Xpn\times_YZ    @>>>     \Xpn \\
@V{\lambda}VV          @VV{\vpn}V \\
Z    @>{\sigma}>>     Y\ . 
\end{CD}
\]
Then, by Definition~\ref{def:vert-element}, $\FpnR\tensR S$ has a vertical element over $R$ if and only if there is a non-zero $m\in\FpnR\tensR S$ such that the zero-set germ $\VV(\Ann_{\OO_{\Xpn\times_YZ,(\xipn,\zeta)}}(m))$ is mapped by $(\sigma\circ\lambda)_{(\xipn,\zeta)}$ to a nowhere-dense subgerm of $Y_{\eta}$. But $\sigma_\zeta$ is a dominant morphism of irreducible germs of the same dimension, and hence $\sigma_\zeta(W_\zeta)$ is nowhere-dense if and only if $W_\zeta$ is so. Therefore verticality of $m$ over $R$ is equivalent to saying that $\lambda_{(\xipn,\zeta)}$ maps $\VV(\Ann_{\OO_{\Xpn\times_YZ,(\xipn,\zeta)}}(m))$ to a nowhere-dense subgerm of $Z_\zeta$; i.e., $m$ is vertical over $S$.
\qed

\section{Algebraic case}
\label{sec:alg}

The following is an algebraic version of our flatness criterion.

\begin{theorem}
\label{thm:main-alg}
Let $R$ be an $n$-dimensional finite type $\C$-algebra, which is locally analytically irreducible.
Let $A$ denote an $R$-algebra essentially of finite type, and let $F$ denote a finitely generated $A$-module. Let $S$ be any $n$-dimensional regular $R$-algebra of finite type such that the induced morphism $\Spec S\to\Spec R$ is dominant. Then $F$ is $R$-flat if and only if the tensor product
\[
\underbrace{F\otimes_R\dots\otimes_RF}_{n\ \mathrm{times}}\otimes_RS
\]
is a torsion-free $R$-module (equivalently, a torsion-free $S$-module).
\end{theorem}

Here, by a \emph{locally analytically irreducible} algebra we mean a finite type $\C$-algebra $R$ such that the complex-analytic space canonically associated with $\Spec R$ is locally irreducible. This is the case, for example, if $R$ is normal. An \emph{$R$-algebra essentially of finite type} means a localization of a finite type $R$-algebra.
\medskip

By analogy with Corollary~\ref{cor:X-smooth}, if $F=A$ is itself regular $n$-dimensional, we get the following:

\begin{corollary}
\label{cor:A-regular}
Let $R$ be an $n$-dimensional finite type $\C$-algebra, which is locally analytically irreducible, and let $A$ be a regular $n$-dimensional $R$-algebra of finite type such that the induced morphism $\Spec A\to\Spec R$ is dominant. Then $A$ is $R$-flat if and only if the $(n+1)$-fold tensor power $A^{\otimes^{n+1}_R}$ is a torsion-free $R$-module.
\end{corollary}

\begin{remark}
\label{rem:effective}
By Theorem~\ref{thm:main-alg}, in order to verify that $F$ is not $R$-flat, it suffices to find an associated prime $\pp$ of $F^{\otimes^n_R}\otimes_RS$ in $A^{\otimes^n_R}\otimes_RS$ for which $\pp\cap R\neq(0)$. Thus our criterion paired with computer algorithms for primary decomposition and desingularization (see, e.g., \cite{GP}) provides a tool for checking flatness by effective computation, over an arbitrary complex affine locally analytically irreducible domain.

On the other hand, the bound $n=\dim R$ is sharp (see Remark~\ref{rem:generalization} and Example~\ref{ex:sharp}).
\end{remark}

The proof of Theorem~\ref{thm:main-alg} is virtually identical with that of \cite[Thm.\,1.3]{ABM1} (an algebraic variant of the flatness criterion of \cite{ABM1}).
One reduces to the analytic case by considering the complex-analytic spaces canonically associated with the spectra of the given rings, and then uses standard faithfull flatness arguments and the fact that vertical elements are precisely the zero-divisors in the algebraic case. For details we refer the reader to \cite{ABM1}.

\begin{remark}
\label{rem:char-zero}
It is worth pointing out that, by the Tarski-Lefschetz Principle, one can generalize Theorem~\ref{thm:main-alg} by replacing $\C$ with any field of characteristic zero (see~\cite{ABM2} for details).
In this case, ``analytically irreducible'' needs to be replaced with ``geometrically unibranch'' (cf.~\cite{Gro}).
\end{remark}
\medskip

We conclude with an explicit calculation showing Theorem~\ref{thm:main-alg} at work.
This example proves as well that, in general, over a non-regular $R$, tensoring with $S$ is necessary for detecting non-flatness of an $R$-module $F$.

\begin{example}
\label{ex:sharp}
The following example of a non-flat module is due to Douady (see, e.g., \cite[\S\,3.13]{F}).
Let
\[
R:=\frac{\C[y_1,y_2]}{(4y_1^3+27y_2^2)}\,,
\]
and
\[
F:=\frac{\C[y_1,y_2,x]}{\sqrt{(4y_1^3+27y_2^2, x^3+y_1x+y_2)}}\,,
\]
where $\sqrt{I}$ denotes the radical of $I$. We will use Theorem~\ref{thm:main-alg} to verify that $F$ is a non-flat $R$-module.

Consider the (locally analytically irreducible) plane curve defined by $R$, and let $S=\C[u]$ denote the coordinate ring of its normalization. Then $S$ is a regular $R$-algebra, of dimension $\dim S=1=\dim R$, and the mapping $\Spec S\to\Spec R$ is dominant. Now, since $\dim R=1$, we need to look for $R$-torsion in $F\otimes_RS$. Note that, as an $R$-module, $S\cong\C[y_1,y_2,u]/(y_1+3u^2, y_2-2u^3)$, and hence
\[
F\otimes_RS\ \cong\ \frac{\C[y_1,y_2,x,u]}{\sqrt{(4y_1^3+27y_2^2, x^3+y_1x+y_2)}+(y_1+3u^2, y_2-2u^3)}\,.
\]
One can use a computer algebra software, like Singular (cf.~\cite{GP}), to verify that the $S$-module $F\otimes_RS$ has an associated prime $\pp$ such that $\pp\cap R=(y_1,y_2)$. Therefore $F\otimes_RS$ has non-zero $R$-torsion, and hence $F$ is not $R$-flat, by Theorem~\ref{thm:main-alg}.
On the other hand, notice that $F$ itself is a torsion-free $R$-module. To see this, observe that $F$ has precisely two associated primes each of which contracts to zero in $R$.

Note also that $S$ itself is not $R$-flat. This can be verified (in light of Theorem~\ref{thm:main-alg}) by looking at the $R$-module structure of $S\otimes_RS$. One can readily see that the element $u-t$ is a zero-divisor in
\[
S\otimes_RS\ \cong\ \frac{\C[y_1,y_2,u,t]}{(y_1+3u^2,y_2-2u^3,y_1+3t^2,y_2-2t^3)}\,.
\]
Therefore $S$ is not $R$-flat, by Theorem~\ref{thm:main-alg} (or Corollary~\ref{cor:X-smooth}). On the other hand, $S$ itself is $R$-torsion-free, which proves sharpness of our flatness criterion.
\end{example}


\bibliographystyle{amsplain}

\end{document}